\documentclass{article}
\pdfoutput=1
\usepackage[T1]{fontenc}
\usepackage[cp1250]{inputenc}
\usepackage{amsthm}
\usepackage{amsmath}
\usepackage{amsfonts}
\usepackage{multicol}
\usepackage{epsfig}
\usepackage{verbatim}
\usepackage[affil-it]{authblk}
\usepackage[authoryear]{natbib}
\usepackage{subcaption}
\usepackage{amssymb}
\usepackage{pdfpages}

\usepackage{hyperref}
\hypersetup{
	bookmarksnumbered,
	linkcolor=black,
	citecolor=black,
	colorlinks=true,
}%

\newtheorem{definition}{Definition}
\newtheorem{theorem}{Theorem}

\theoremstyle{definition}
\newtheorem{example}{Example}

\theoremstyle{plain}

\newtheorem{proposition}{Proposition}

\def\R{\mathbb R}

\def\C{\mathcal C}
\def\N{\mathcal N}

\def\S{\mathcal S}
\def\E{\mathcal E}

\def\diag{\mathrm{diag}}
\def\Var{\mathrm{Var}}

\title{On weighted optimality of experimental designs}
\author{Samuel Rosa}
\affil{Faculty of Mathematics, Physics and Informatics, Comenius University, Bratislava, Slovakia}
\date{\today} 

\begin{document}
	
\maketitle

\begin{abstract}
	When the experimental objective is expressed by a set of estimable functions, and any eigenvalue-based optimality criterion is selected, we prove the equivalence of the recently introduced weighted optimality and the 'standard' optimality criteria for estimating this set of functions of interest. Also, given a weighted eigenvalue-based criterion, we construct a system of estimable functions, so that the optimality for estimating this system of functions is equivalent to the weighted optimality. This allows one to use the large body of existing theoretical and computational results for the standard optimality criteria for estimating a system of interest to derive theorems and numerical algorithms for the weighted optimality of experimental designs. 
	Moreover, we extend the theory of weighted optimality so that it captures the experimental objective consisting of any system of estimable functions, which was not the case in the literature on weighted optimality so far. For any set of estimable functions, we propose a corresponding weight matrix of a simple form, and with a straightforward interpretation. Given a set of estimable functions with their corresponding weights, we show that it is useful to distinguish between the primary weights selected by the experimenters and the secondary weights implied by the weight matrix.
\end{abstract}

\section{Introduction}

In a given experiment, certain parameters or functions of parameters may be of greater interest than others. Then, the experimental design should reflect this. In the optimal design theory, such objectives can be expressed using the weighted optimality criteria, developed by \cite{MorganWang10} and \cite{MorganWang}. These authors consider models with treatment effects in which various weights are placed on the particular treatments. They build a diagonal weight matrix, and propose a weighted information matrix that captures the given weights.

However, \cite{StallingsMorgan} found the approach of Morgan and Wang too restrictive, and therefore extended the theory of the weighted optimality. They propose a weighted information matrix for any, generally non-diagonal, positive definite weight matrix. Such an approach allows one to capture situations, in which the experimental objective is to estimate a given set of estimable functions (possibly with assigned various weights).

Although the weighted optimality criteria provide a meaningful approach of expressing the interest in a given set of linear estimable functions, such an objective can also be expressed directly, using the theory of optimality for estimating the given functions of interest (e.g., see \cite{puk}). Let us denote the system of functions of interest as $Q^T\tau$, where $Q$ is a given coefficient matrix and $\tau$ is the vector of the effects of interest. Then, the theory of optimality for $Q^T\tau$ makes use of the information matrix for $Q^T\tau$, as opposed to the information matrix for the vector $\tau$.

The optimality for a system of interest is a well developed theory with a wide range of known results. Many of these results deal generally with the optimality criteria - notably the equivalence theorems (see Chapter 7 by \cite{puk}); and many more provide optimal designs for particular models, systems of estimable functions and optimality criteria. In blocking experiments, e.g., see \cite{MajumdarNotz} for comparing test treatment with a control, and \cite{Majumdar86} for comparing two sets of treatments; see \cite{RosaHarman16} for optimal designs in general models with treatment effects and nuisance effects. There is also a multitude of results on algorithms for constructing optimal (or efficient) designs for a system of functions of interest: e.g., rounding of optimal approximate designs (Chapter 12 by \cite{puk}), multiplicative algorithms (\cite{Yu10}), and mixed integer second order cone programming (\cite{SagnolHarman15}).
\bigskip

In this paper, we build upon the theory presented by \cite{StallingsMorgan}. In Section \ref{sOptimalityForQ}, we compare the weighted optimality and the 'standard' optimality for estimating a system of interest. We show that given a system of functions of interest $Q^T\tau$, the corresponding weighted optimality and the optimality for $Q^T\tau$ are equivalent for any eigenvalue-based criterion. Furthermore, we note that the weights assigned to the functions $Q^T\tau$ can be expressed by appropriately rescaling these functions. Note that most of the well-known criteria are eigenvalue-based; e.g., the $D$-, $A$- and $E$-optimality all depend only on the positive eigenvalues of the information matrix. 
We also consider the situation, where a weight matrix $W$ is selected, instead of a system of functions of interest. In such a case, we provide a method for constructing a system of functions of interest $Q^T\tau$ for the given weight matrix, so that for any eigenvalue-based criterion the optimality for $Q^T\tau$ is equivalent to the weighted optimality given by $W$.

When one considers a weighted optimality criterion, the provided equivalence with the optimality for a system of interest allows one to employ the tools of the extensive theory of the 'standard' optimality for a system of interest, e.g., the equivalence theorems and the algorithms for computing optimal exact designs.


Moreover, we extend the theory of the weighted optimality. The weight matrices corresponding to a system $Q^T\tau$ are defined by \cite{StallingsMorgan} only for systems that satisfy that $\mathrm{rank}(Q)$ is at least equal to the dimension of the set of all estimable functions. This, for instance, does not cover the case of $c$-optimality (see \cite{puk}, Chapter 2).
We propose to relax the definition of the weight matrix so that a weight matrix can be constructed for any system of estimable functions. Moreover, we define a weight matrix corresponding to a system $Q^T\tau$, which is of a slightly simpler form than that by \cite{StallingsMorgan}, although these two forms are equivalent with respect to the implied weighting of estimable functions. Finally, we note that it is advisable to differentiate between primary weights specified by the experimenters and secondary weights given by the weight matrix that corresponds to the system of interest $Q^T\tau$.

\subsection{Notation}

We define $1_n$ and $0_n$ as the column vectors of length $n$ of all ones and zeroes, respectively. The symbol $I_n$ denotes the identity matrix, and we define $J_n := 1_n 1_n^T$ and $0_{m \times n} := 0_m 0_n^T$. 
By $\mathrm{diag}(x_1,\ldots,x_n)$ we denote the $n \times n$ diagonal matrix with $x_1, \ldots, x_n$ on diagonal. The symbols $\mathfrak{S}^n_+$ and $\mathfrak{S}^n_{++}$ denote the sets of $n \times n$ non-negative definite matrices and positive definite matrices, respectively. 
Given $A\in \S^n_+$, we denote its ordered eigenvalues by $\lambda_1(A) \geq \ldots \geq \lambda_n(A)$. By $\C(A)$ and $\N(A)$ we denote the column space and the null space of the matrix $A$, respectively. Given a matrix $A$, the Moore-Penrose pseudoinverse of $A$ is denoted by $A^+$, and $A^-$ is any generalized inverse of $A$. By $A^{+1/2}$ we denote the matrix $(A^+)^{1/2}$.

\subsection{The model}

Consider the model studied by \cite{StallingsMorgan}:
\begin{equation}\label{eModel1}
y = X(\xi)\tau + L\beta + \varepsilon,
\end{equation} 
where $y=(y_1,\ldots,y_n)^T$ is the response, $\tau=(\tau_1, \ldots, \tau_v)^T$ is the vector of the parameters of interest, $\beta=(\beta_1,\ldots,\beta_m)^T$ is the vector of nuisance parameters, and $\varepsilon$ is the $n \times 1$ vector of uncorrelated errors with $E(\varepsilon)=0_n$ and $\Var(\varepsilon)=\sigma^2 I_n$. The (exact) design $\xi$ determines the matrix $X(\xi)$ that relates $\tau$ to $y$, and $L$ is the matrix relating $\beta$ to $y$, which is independent of the choice of the design $\xi$. Since the results in this paper do not depend on the variance of the errors, for simplicity we assume $\sigma^2=1$.

The information matrix for $\tau$ is $C(\xi)= X^T(\xi)(I-P_L)X(\xi)$, where $P_L=L(L^TL)^-L^T$ is the matrix of the orthogonal projection on $\C(L)$. A function $q^T \tau$ is estimable under a design $\xi$ if and only if $q \in \C(C(\xi))$.

As in \cite{StallingsMorgan}, we require for all competing designs to be able to estimate the same set of linear functions of the parameters $q^T\tau$. We define $\E$ as the estimation space, which is the set of all vectors $q$ corresponding to these estimable functions $q^T\tau$. That is, all competing designs have the same $\C(C(\xi))$, namely $\C(C(\xi))=\E$. Often, the estimation space $\E$ is the entire $\R^v$ or the set of all vectors orthogonal to $1_v$, denoted by $1_v^\perp$. We denote the $v \times v$ matrix of the orthogonal projection on $\E$ by $P_\tau$.

A common example of model \eqref{eModel1} is a model, where $\tau$ is the vector of treatment effects, and $\beta$ are some nuisance effects. In this case, the estimation space is $\E = 1_v^\perp$, which represents the set of all treatment contrasts.

We consider a system of $s$ estimable functions $Q^T\tau$, where $Q$ is a $v \times s$ matrix of rank $r$, satisfying $\C(Q) \subseteq \E$; we denote the columns of $Q$ as $q_1, \ldots, q_s$, and the elements of $Q$ as $Q_{ij}$.  We say that a design $\xi$ is feasible for $Q^T\tau$ if $Q^T\tau$ is estimable under $\xi$. 
We say that the system $Q^T\tau$ is unscaled (or normalized) if $\lVert q_i \rVert = 1$ for all $i$; otherwise, $Q^T\tau$ is said to be scaled.
Furthermore, if $r=s$, then we say that the system $Q^T\tau$ is of full rank; otherwise, $Q^T\tau$ is said to be rank-deficient. Note that for a system $Q^T\tau$ to attain the full rank, it is necessary that $s\leq\dim(\E)$.

If the interest in each of the functions in $Q^T\tau$ is not the same, we represent the various interest by weights $b_1, \ldots, b_s$ of the functions $q_1^T\tau, \ldots, q_s^T\tau$. We denote the matrix of weights of the functions of interest as $B=\diag(b_1,\ldots, b_s)$.

\subsection{Weighted optimality}\label{ssWeightedOpt}

For now, let us consider the weighted optimality criteria, as defined by \cite{StallingsMorgan}. That is, we say that the weight matrix is any $v \times v$ positive definite matrix $W$. For a given $W$ and any estimable function $q^T\tau$, the weight assigned to $q^T\tau$ is $(q^T W^{-1} q)^{-1}$ and the corresponding weighted variance is $\Var_W(\widehat{q^T\tau}) = (q^T W^{-1} q)^{-1} \Var(\widehat{q^T\tau})$.

Then, the weighted information matrix of a feasible design $\xi$ (given a weight matrix $W$) is defined as $C_W(\xi) = W^{-1/2} C(\xi) W^{-1/2}$. Given an optimality criterion $\Phi$, a design is $\Phi$-optimal with respect to $W$ (or, in short, $\Phi_W$-optimal) if it maximizes $\Phi_W(C(\xi)):=\Phi(C_W(\xi))$. Note that \cite{StallingsMorgan} consider optimality criteria that are to be minimized. The correspondence between the weighted variance and the weighted information matrix is represented by Lemma 1 by \cite{StallingsMorgan}, which states that the weighted variance of a given estimable function $q^T\tau$ is a convex combination of the inverses of the eigenvalues of $C_W(\xi)$.

Two weight matrices $W_1$ and $W_2$ are said to be estimation equivalent if $q^T W_1^{-1} q = c q^T W_2^{-1} q$ for all $q \in \E$ and for some constant $c$. Lemma 3 by \cite{StallingsMorgan} shows that $W_1$ and $W_2$ are estimation equivalent if and only if $P_\tau W_1^{-1} P_\tau = c P_\tau W_2^{-1} P_\tau$ for some constant $c$.
\bigskip

When the objective of the experiment is to estimate a set of $s$ normalized functions $Q^T\tau$, where $s$ is equal to the dimension of $\E$ and $r=s$, \cite{StallingsMorgan} propose a corresponding weight matrix $W_Q=I-P_\tau + QQ^T$, where $P_\tau$ is the orthogonal projector on $\E$. The inverse of such a matrix is $W_Q^{-1} = I-P_\tau + (QQ^T)^+$, and $W_Q$ places weight 1 on each of the functions of interest $q_i^T\tau$, i.e., $(q_i^T W_Q^{-1} q_i)^{-1} = 1$ for $i=1,\ldots,s$. In particular, it holds that $Q^T W_Q^{-1} Q = I_s$. If the interest in the functions $q_1^T\tau, \ldots, q_s^T\tau$ is represented by the weights $b_1, \ldots, b_s$, one can employ the weighted version of $Q$ given by $\tilde{Q}:=QB^{1/2}$. That is, the interest $b_i$ in $q_i^T\tau$ is expressed by rescaling the function of interest using $\tilde{q}_i=b_i^{1/2}q_i$, so that $\lVert \tilde{q}_i \rVert = b_i^{1/2}$.  Then, the corresponding weight matrix $W_{\tilde{Q}}$ places weight $b_j$ on $q_j^T\tau$, $j=1,\ldots, s$.

When the normalized system of contrasts $Q^T\tau$ satisfies $r=\dim(\E)$ and $s>r$ (i.e., when we have 'too many' functions of interest), the authors employ the same form of the corresponding weight matrix $W_Q$ as in the case $r=s=\dim(\E)$. However, in general, such a matrix does not place equal weights on each of the functions $q_i^T \tau$. That is not a desirable property, because these functions of interest should be assigned uniform weights. However, \cite{StallingsMorgan} show that the desired weights can be captured by the weighted $A$-optimality criterion, even with such a weight matrix.

\section{Optimality for a system of interest}\label{sOptimalityForQ}

\cite{StallingsMorgan} observed that the standard optimality criteria defined on the information matrices $C(\xi)$ place equal emphasis on all normalized estimable functions, which is an evident limitation of such approach. Consequently, the weighted optimality criteria present a way of eliminating such limitations, because they allow for various emphases on the estimable functions.

However, this limitation of the standard optimality criteria arises primarily not because of the criteria considered, but rather it is due to the narrow definition of the information matrix $C(\xi)$, which captures the information about exactly the vector $\tau$. When the interest is in a set of $s$ estimable functions $Q^T\tau$, where $r=s$, this limitation can be completely circumvented by considering the well-established \emph{information matrix} of a feasible design $\xi$ \emph{for the system of interest} $Q^T\tau$ (as opposed to the information matrix $C(\xi)$ for the vector $\tau$):
$$N_Q(\xi) = (Q^T C^-(\xi) Q)^{-1},$$
see \cite{puk}. Such an information matrix is proportional to the inverse of the variance matrix for the least-squares estimator of $Q^T\tau$.
Given a standard optimality criterion, a design is said to be \emph{$\Phi$-optimal for the system of interest $Q^T\tau$} (or $\Phi_Q$-optimal) if it maximizes $\Phi(N_Q(\xi))$.

In the second case in which \cite{StallingsMorgan} defined the matrix $W_Q$, i.e., when $s>r=\dim(\E)$, the information matrix for $Q^T\tau$ of the form $(Q^T C^-(\xi) Q)^{-1}$ is not well defined, as the inverse matrix does not exist. Instead of the inverse, one can consider the pseudo-inverse, resulting in the information matrix 
\begin{equation}\label{eInfMat}
N_Q(\xi) = (Q^T C^-(\xi) Q)^{+}
\end{equation}
for $s>r$ (e.g., see \cite{puk}). Note that the form \eqref{eInfMat} encompasses both the cases $s=r$ and $s>r$. Also note that the information matrix $N_Q(\xi)$ is well defined even for $r<\dim(\E)$.

We show that not only can the experimental objective of estimating a set of estimable functions $Q^T\tau$ be represented both by the corresponding weighted optimality as well as by the information matrix for $Q^T\tau$; these approaches are in fact equivalent for any criterion based only on the (non-negative) eigenvalues of the information matrix, as is proved in the following theorem. Note that the weighted optimality of \cite{StallingsMorgan} is built primarily for the eigenvalue-based optimality criteria, e.g., see Lemma 1 therein.

\begin{theorem}\label{tEigEquiv}
	Let $Q^T\tau$ be a system of estimable functions that satisfies $r=\dim(\E)$, and let $\xi$ be a feasible design for $Q^T\tau$. Then, the information matrix $N_Q(\xi) =  (Q^T C^-(\xi) Q)^{+}$ for $Q^T\tau$ and the weighted information matrix $C_{W_Q}(\xi) = W_Q^{-1/2} C(\xi) W_Q^{-1/2}$, where $W_Q = I-P_\tau + (QQ^T)$, have the same non-zero eigenvalues, including multiplicities.
\end{theorem}

\begin{proof}
	Let us denote $C:= C(\xi)$. In this proof, we will employ the facts that the non-zero eigenvalues of $AB$ are identical to the eigenvalues of $BA$ if $BA$ is defined, and that the non-zero eigenvalues of $AA^T$ are identical to the eigenvalues of $A^T A$, see 6.54 by \cite{Seber}.
	
	The eigenvalues of $C_{W_Q}(\xi)$ are identical to the eigenvalues of $W_Q^{-1} C$, which can be expressed as $(I-P_\tau + (QQ^T)^+ ) C = (QQ^T)^+ C$, where the equality follows from $\C(C) \subseteq \E$. Therefore, the non-zero eigenvalues of $C_{W_Q}(\xi)$ are inverse to the eigenvalues of $((QQ^T)^+ C)^+$.
	
	Theorem 2 of \cite{Greville66} implies that for two symmetric matrices $A$ and $B$ of the same size, the formula $(AB)^+ = B^+ A^+$ holds if and only if $A^+A B B^T$ and $A^T A B B^+$ are symmetric. Because $r=\dim(\E)$ and $\C(Q) \subseteq \C(C) = \E$, it follows that $\C(Q) = \C(C) = \E$. Moreover, $\C((QQ^T)^+) = \C(QQ^T) = \C(Q) = \C(C) = \C(C^+)$, which yields that there exist matrices $Z$ and $V$, such that $C=QQ^TZ$ and $(QQ^T)^+ = C^+ V$. Since both $(QQ^T)^+$ and $C^+$ are symmetric, it follows that $(QQ^T)^+  = V^T C^+$. By setting $A= (QQ^T)^+$ and $B=C$ in Theorem 2 of \cite{Greville66}, we obtain 
	$$A^+A B B^T = QQ^T (QQ^T)^+ C C^T =   QQ^T (QQ^T)^+ QQ^T Z C^T = C C^T$$
	and
	$$A^T A B B^+ = (QQ^T)^+ (QQ^T)^+ C C^+ = (QQ^T)^+ V^T C^+ C C^+ = (QQ^T)^+ (QQ^T)^+, $$
	which are both evidently symmetric. Therefore, the non-zero eigenvalues of $C_{W_Q}(\xi)$ are inverse to the eigenvalues of $C^+ QQ^T$, which has the same non-zero eigenvalues as $Q^T C^+ Q$. Thus, the matrix $C_{W_Q}(\xi)$ has the same non-zero eigenvalues as $(Q^T C^+ Q)^+ = N_Q(\xi)$.
\end{proof}

\begin{figure}[h]
	\begin{center}
	\includegraphics[trim=5.3cm 16.5cm 6.0cm 4.3cm, clip]{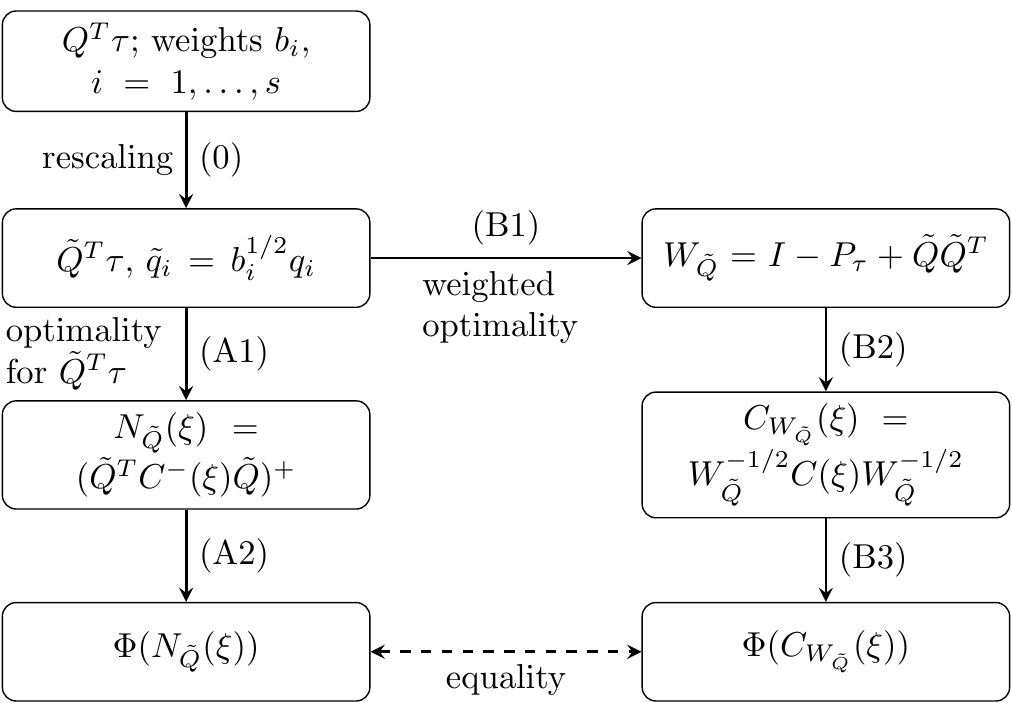}
	\caption{Constructing optimality criteria. The system of estimable functions $Q^T\tau$ with weights $b_i$ is rescaled in step (0). Suppose that $\xi$ is a feasible design. The weighted optimality of \cite{StallingsMorgan} is obtained by constructing the weight matrix $W_{\tilde{Q}}$ (B1), then the weighted information matrix $C_{W_{\tilde{Q}}}$ (B2), and finally, applying the optimality criterion to the weighted information matrix (B3). Alternatively, the optimality for the given weighted system is obtained by constructing the information matrix for $\tilde{Q}^T\tau$ (A1) and applying the optimality criterion to this information matrix (A2). Theorem \ref{tEigEquiv} shows the equivalence of these two approaches, i.e., that $\Phi(N_Q(\xi))$ is equal to $\Phi(C_{\tilde{W}}(\xi))$ for any eigenvalue-based criterion $\Phi$.}\label{fEigEquiv}
	\end{center}
\end{figure}

The construction of and the relationship between the weighted optimality and the optimality for $Q^T\tau$ is depicted in Figure \ref{fEigEquiv}.
The condition $r=\dim(\E)$ in Theorem \ref{tEigEquiv} is necessary for the weight matrix $W_Q$ (as defined by \cite{StallingsMorgan}) to exist. Unlike the weighted information matrix $C_Q$, the information matrix for $Q^T\tau$ allows one to consider also the case in which $r<\dim(\E)$. In Sections \ref{sExtendingWeights} and \ref{sWeightsForContrasts}, we will propose a weight matrix that captures also $r < \dim(\E)$.

The coefficient matrix $Q$ in Theorem \ref{tEigEquiv} can be either scaled or unscaled. Therefore, even when considering a system of contrasts $Q^T\tau$ with non-uniform weights $b_i>0$, $i=1,\ldots,s$, by Theorem \ref{tEigEquiv}, such an objective can be captured by the information matrix for some system of estimable functions. Namely, such a weighting  can be accomplished by employing the matrix $N_{\tilde{Q}}$ for the weighted system of estimable functions $\tilde{Q}^T\tau=B^{1/2}Q^T\tau$. In other words, the weight placed on a given estimable function $q_i^T\tau$ can be expressed by appropriately rescaling the function to $b_i^{1/2}q_i^T\tau$.

An additional evidence, besides the results mentioned in Introduction, that the theory of optimality for $Q^T\tau$ is well developed is, e.g., the fact that Theorem 2 by \cite{StallingsMorgan} is a corollary of Corollary 8.8 by \cite{puk} in light of Theorem \ref{tEigEquiv}. Moreover, the optimality for $Q^T\tau$ eliminates the need for the intermediary step of constructing the weight matrix $W_Q$ from a given system of estimable functions; instead, the matrix $N_Q(\xi)$ allows one to consider the system of interest $Q^T\tau$ directly, see Figure \ref{fEigEquiv}. 
\bigskip

We have obtained that when the experimental objective is to estimate a system of estimable functions $Q^T\tau$ with given weights, then such an objective can be equivalently expressed directly by the information matrix for $Q^T\tau$, instead of constructing the weight matrix and then considering the weighted optimality. However, one may specify the weight matrix (as in \cite{MorganWang10}), instead of considering a set of weighted estimable functions. This brings up an inverse problem: can a system of estimable functions that corresponds to the pre-specified weight matrix be constructed? We provide a positive answer to this question. First, let us define the $v \times v$ matrix $R:=(P_\tau W^{-1} P_\tau)^{+1/2}$. Such a matrix satisfies $\C(R) = \E$, $RR^T = R^2 = (P_\tau W^{-1} P_\tau)^+$, and $R^T\tau=R\tau$ is the sought system of estimable functions.

\begin{theorem}\label{tInverseProblem}
	Let $W \in \S^v_{++}$ be a weight matrix, let $R=(P_\tau W^{-1} P_\tau)^{+1/2}$. Then, given a design $\xi$ that is feasible for $R\tau$, the weighted information matrix $C_W(\xi)$ has the same eigenvalues (including multiplicities) as the information matrix $N_{R}(\xi)$ for $R\tau$.
\end{theorem}

\begin{proof} 
	In this proof, we employ the same techniques as in the proof of Theorem \ref{tEigEquiv}. Let us denote $C:= C(\xi)$ and $W_2:= P_\tau W^{-1} P_\tau$.
	The eigenvalues of $C_{W}(\xi)$ are identical to the eigenvalues of 
	$W^{-1} C = W^{-1} P_\tau C P_\tau$, which are in turn identical to the eigenvalues of $P_\tau W^{-1} P_\tau C = W_2 C$. The non-zero eigenvalues of $W_2 C$ are inverse to the non-zero eigenvalues of $(W_2 C)^+$.
	
	Note that $\C(W_2^+) = \C(W_2) = \E = \C(C) = \C(C^+)$. Therefore, there exist some matrices $Z$ and $H$, such that $W_2 = C^+ Z$ and $C = W_2^+H$. Transposing the former equation yields $W_2 = Z^T C^+$.
	Let us employ Theorem 2 by \cite{Greville66} for $A=W_2$ and $B=C$. Then, 
	$$A^T A B B^+ = W_2 W_2 C C^+ = W_2 Z^T C^+ C C^+ = W_2 Z^T C^+ = W_2^2$$
	and
	$$A^+ A B B^T = W_2^+ W_2 C C =  W_2^+ W_2 W_2^+ H C = W_2^+ H C = C^2,$$
	which are both symmetric. Thus, $(W_2 C)^+ = C^+ W_2^+ = C^+ RR$. The eigenvalues of $C^+ RR$ are identical to the eigenvalues of $RC^+ R$, which are in turn inverse to the eigenvalues of $(RC^+ R)^+ = N_R(\xi)$. This shows that the non-zero eigenvalues of $C_W(\xi)$ and $N_R(\xi)$ are identical.
	Moreover, both $N_R(\xi)$ and $C_W(\xi)$ are $v \times v$ matrices, thus they even have the same multiplicity of the zero eigenvalue.
\end{proof}
\begin{figure}[h]
	\begin{center}
		\includegraphics[trim=5.3cm 18.3cm 5.5cm 4.3cm, clip]{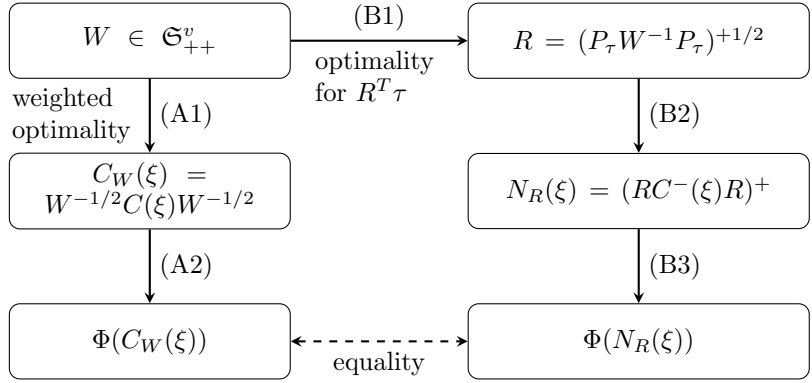}
		\caption{Constructing optimality criteria for a given weight matrix. Suppose that $\xi$ is a feasible design. The weighted optimality is obtained by constructing the weighted information matrix $C_{W}(\xi)$ (A1) and then applying the optimality criterion to the weighted information matrix (A2). Alternatively, the optimality for a system of estimable functions is obtained by constructing the sought system $R\tau$ (B1), constructing the information matrix for $R\tau$ (B2) and then applying the optimality criterion to this information matrix (B3). Theorem \ref{tInverseProblem} shows the equivalence of these two approaches, i.e., that $\Phi(C_W(\xi))$ is equal to $\Phi(N_R(\xi))$ for any eigenvalue-based criterion $\Phi$.}\label{fEigEquiv2}
	\end{center}
\end{figure}

By Theorem \ref{tInverseProblem}, weighted optimality with respect to any eigenvalue-based criterion can be transformed to optimality for $R\tau$; and thus, already developed algorithms for optimality for estimating a system of estimable functions can be employed. Figure \ref{fEigEquiv2} illustrates the construction of the weighted optimality criterion for a given $W$ and of the optimality criterion for the corresponding system $R\tau$.

\section{Extending the definition of the weight matrix}\label{sExtendingWeights}

When $\dim(\E)<v$, one needs not weight all the functions with coefficients in $\R^v$; e.g., in a model with treatment effects and nuisance effects, there is no point in weighting functions other than the treatment contrasts. Moreover, the experimenters may wish to assign non-zero weights only to some subset of the estimable functions. Then, we propose to express the set of all coefficient vectors of functions that are to be weighted using $\C(W)$. In particular, we require only for the functions $q^T\tau$ with $q \in \C(W)$ to be weighted.

Therefore, we propose that the weight matrices need not be non-singular; in fact, since $\C(W)$ represents the set of functions that are to be weighted, ranks of the proposed weight matrices are equal to the dimensions of the sets of the considered coefficient vectors $q$. For example, in a model with treatment and nuisance effects, it always holds that $\C(W) \subseteq 1_v^\perp$ and thus the proposed weight matrix satisfies $\mathrm{rank}(W) \leq v-1$ in such a model. Clearly, for singular $W$, the assigned weights cannot be given by $(q^T W^{-1} q)^{-1}$. A natural relaxation of this formula is $(q^T W^- q)^{-1}$. Then, given a weight matrix $W$ and a coefficient vector $q \in \C(W)$, we define the weight of $q^T\tau$ as $(q^T W^- q)^{-1}$; if $q \not \in \C(W)$, the weight of $q^T\tau$ is zero. 

\begin{definition}\label{dWeightMatrix}
	(i) The matrix $W$ is a weight matrix if  $W \in \S^v_+$ and $\C(W) \subseteq \E$. 
	(ii) Let $W$ be a weight matrix and let $q \in \C(W)$. Then, the weighted variance of $q^T\tau$ is 
	$$\Var_W(\widehat{q^T\tau}) = (q^T W^- q)^{-1} \Var(\widehat{q^T\tau}).$$
\end{definition}


Note that the condition $q \in \C(W)$ ensures that the weight of $q^T\tau$ does not depend on the choice of the generalized inverse $W^-$. Moreover, if $q \in \C(W)$, then $q=Wh$ for some $h \in \R^v$; thus $q^T W^- q = h^T W h$, which is equal to zero if and only if $h \in \N(W)$, i.e., when $q=Wh=0_v$. It follows that $W$ places zero weight on an estimable function that satisfies $q \in \C(W)$ if and only if $q=0_v$.

Following Definition \ref{dWeightMatrix}, we must appropriately define the weighted information matrix, as the definition $C_W(\xi) = W^{-1/2} C(\xi) W^{-1/2}$ is not feasible for singular $W$. Given a weight matrix $W$, we denote its spectral decomposition by $W=S_W\Lambda_W S_W^T$, where $\Lambda_W=\diag(\lambda_1(W), \ldots, \lambda_v(W))$ and $S_W$ is a $v \times v$ orthogonal matrix. We denote the rank of $W$ by $d$. Then, if $d<v$, we may also write $W=F D F^T$, where $D=\diag(\lambda_1(W), \ldots, \lambda_d(W))$ is a $d \times d$ matrix of the positive eigenvalues of $W$, and the $v \times d$ matrix $F$ contains the first $d$ columns of $S_W$. Then, $W$ can be expressed as $W=K_W K_W^T$, where $K_W=FD^{1/2}$ is a $v \times d$ matrix of rank $d$.

\begin{definition}\label{dWeightedInformation}
	(i) Let $W$ be a weight matrix and let $\xi$ be a design that satisfies $\C(W) \subseteq \C(C(\xi))$. Then, the weighted information matrix of $\xi$ is
	$$C_W(\xi) = (K_W^T C^-(\xi) K_W)^{-1},$$
	where $K_W$ is defined in the previous paragraph.
	
	(ii) Let $\Phi$ be an optimality criterion. Then, a design is $\Phi$-optimal with respect to $W$, or $\Phi_W$-optimal in short, if it maximizes $\Phi(C_W(\xi))$.
\end{definition}

The condition  $\C(W) \subseteq \C(C(\xi))$ ensures that the functions with coefficient vectors $q \in \C(W)$ that are to be weighted are actually estimable under $\xi$, i.e., $q \in \C(C(\xi))$.
The $v \times d$ matrix $K_W$ in Definition \ref{dWeightedInformation} satisfies $K_WK_W^T = W$, which indicates that it is a quasi-square root of $W$, with the number of columns adjusted to correspond to its rank; compare to $C_W(\xi) = W^{-1/2} C(\xi) W^{-1/2}$ given by \cite{StallingsMorgan}. Note that both the matrices $K_W$ and $C(\xi)$ are of rank $d$, and $\C(K_W) \subseteq \C(C(\xi))$; thus, the weighted information matrix $C_W(\xi)$ is well-defined, and is of full rank.

We remark that the weighted information matrix $C_W(\xi)$ coincides with the information matrix $N_K(\xi)$ for estimating $K_W^T\tau$. Therefore, Definition \ref{dWeightedInformation} in fact consists of finding a set of estimable functions of full rank corresponding to $W$ and then calculating the information matrix for the obtained set.
\bigskip

Similarly to \cite{StallingsMorgan}, we show that the proposed weighted information matrix is relevant with respect to the weighted variances.

\begin{proposition}\label{pEigenVar}
	Let $W$ be a weight matrix of rank $d$, let $q \in \C(W)$ and let $\xi$ be a design that is feasible for $q^T\tau$. Then, the weighted variance of $\widehat{q^T\tau}$ under $\xi$ is a convex combination of $\lambda_1^{-1}(C_W(\xi)), \ldots, \lambda_d^{-1}(C_W(\xi))$.
\end{proposition}

\begin{proof}
	Since $W=K_WK_W^T$, we have $\C(W) = \C(K_W)$ and thus $q\in \C(W) = \C(K_W)$ yields $q=K_Wh$ for some $h \in \R^d$. Therefore,
	$$
	\Var(\widehat{q^T\tau}) = q^T C^-(\xi) q = h^T C_W^{-1}(\xi) h = h^T U \Lambda^{-1} U^T h = \sum_{i=1}^d g_i^2\lambda_i^{-1}(C_W(\xi)),
	$$
	where $U \Lambda U^T$ is the spectral decomposition of $C_W(\xi)$, and $g:= U^T h$. Then, $\sum_i g_i^2 = h^T UU^T h = h^T h$. The weight of $q^T\tau$ is $q^T W^- q = h^T K_W^T (K_WK_W^T)^- K_W h = h^T h$, because $K_W^T (K_WK_W^T)^- K$ is a symmetric idempotent matrix of full rank, i.e., $I_d$. It follows that
	$$
	\Var_W(\widehat{q^T\tau}) = (h^Th)^{-1} \sum_{i=1}^r g_i^2\lambda_i^{-1}(C_W(\xi)) = \sum_{i=1}^r \frac{g_i^2}{g^Tg}\lambda_i^{-1}(C_W(\xi)).
	$$
\end{proof}

Analogously to \cite{StallingsMorgan}, we say that two weight matrices $W_1$ and $W_2$ that satisfy $\C(W_1)=\C(W_2)=:\S$ are \emph{estimation equivalent} if there exists $c > 0$ such that $q^T W_1^- q = c q^T W_2^- q$ for all $q \in \S$.

\begin{proposition}\label{pEstEquiv}
	Let $W_1$ and $W_2$ be weight matrices satisfying $\C(W_1)=\C(W_2)=:\S$. Then, $W_1$ and $W_2$ are estimation equivalent if and only if $P_\S W_1^- P_\S = c P_\S W_2^- P_\S$ for some constant $c>0$, where $P_\S$ is the orthogonal projector on $\S$.
\end{proposition}

\begin{proof}
	The proof is analogous to the proof of Lemma 3 by \cite{StallingsMorgan}. Let $W_1$ and $W_2$ be estimation equivalent. Let us denote $q_i$ as $i$-th column of $P_\S$. Then, $q_i \in \S$ and thus $q_i^T W_1^- q_i = c q_i^T W_2^- q_i$, which shows that the diagonal elements of $P_\S W_k^- P_\S$ satisfy the desired equation. Now, let $q_j$ be the $j$-th column of $P_\S$, $j \neq i$, and let $q:=q_i+q_j \in \S$. Then, $q^T W_1^-q = c q^T W_2^- q$, which yields 
	$$q_i^T W_1^-q_i + q_j^T W_1^-q_j + 2 q_i^T W_1^-q_j = 
	  cq_i^T W_2^-q_i + cq_j^T W_2^-q_j + 2c q_i^T W_2^-q_j,$$
	and thus $q_i^T W_1^- q_j = c q_i^T W_2^- q_j$, which shows the desired result also for non-diagonal elements.
	
	Conversely, let the weight matrices satisfy $P_\S W_1^- P_\S = P_\S W_2^- P_\S$ and let $q \in \S$. Then, $P_\S q = q$ and thus $q^T W_1^- q = q^T P_\S W_1^- P_\S q = c q^T P_\S W_2^- P_\S q = c q^T W_2^- q$.
\end{proof}

Note that because $\C(P_\S) = \C(W_1)=\C(W_2)$, the expressions $P_\S W_1^- P_\S$ and $P_\S W_2^- P_\S$ do not depend on the choice of $W_1^-$ and $W_2^-$, respectively. Furthermore, if $\mathrm{rank}(W_i) = \dim(\E)$, $i=1,2$, then we have $\S=\E$ and thus $P_\S W_1^- P_\S = c P_\S W_2^- P_\S$ becomes $P_\tau W_1^- P_\tau = c P_\tau W_2^- P_\tau$, cf. Lemma 3 by \cite{StallingsMorgan}.

\section{Weights for any set of estimable functions}\label{sWeightsForContrasts}

\subsection{System of a lesser dimension}

Consider the case of $r<\dim(\E)$, i.e., the case in which the coefficient vectors $q_1, \ldots, q_s$ do not span the entire estimation space. The simplest example of such a system of contrasts is a single estimable function $Q^T\tau=q_1^T\tau$, which represents $c$-optimality (where $c=q_1$). As already noted, the experimental settings satisfying $r<\dim(\E)$ are not covered by \cite{StallingsMorgan}. In general, the matrix $W_Q = I-P_\tau + QQ^T$ is not of full rank in such a case; and thus $W_Q^{-1}$, which is needed for the actual weighting, does not exist.

The approach of weighting \emph{any} estimable function $q^T\tau$ by $(q^T W^{-1} q)^{-1}$ may not be necessary if $r<\dim(\E)$, as shown in the following example, which aims to demonstrate this in a simple setting.

\begin{example}
	Consider a simple case of model \eqref{eModel1}, in which $\tau$ is the vector of treatment effects and $\beta$ represents only the constant term, resulting in:
	\begin{equation}\label{eModelTreat}
	y_i = \mu + \tau_{\xi(i)} + \varepsilon_i, \quad i=1,\ldots,n,
	\end{equation}
	where $\mu$ is the overall mean,  and $\xi(i) \in \{1,\ldots,v\}$ is the treatment chosen for the $i$-th trial. The estimation space is $\E = 1_v^\perp$, which represents the set of all treatment contrasts $q^T \tau$.
	
	Suppose that $v=4$ and that we aim to compare only $\tau_2$ and $\tau_1$; thus, we have $Q^T\tau = q_1^T\tau = (\tau_2 - \tau_1)/\sqrt{2}$. Clearly, such a system of contrasts does not place any weight on $\tau_3$ or $\tau_4$. Therefore, any  matrix of treatment weights corresponding to $Q^T\tau$ should not place any weight on any $q^T\tau$, such that $q_3 \neq 0$ or $q_4 \neq 0$.
	
	Now, consider a less extreme case. Suppose that $Q = (q_1, q_2)$, where $q_2^T = (1,1,1,-3)/\sqrt{12}$, and let $q_3^T = (1,1,-2,0)/\sqrt{6}$. Then, the coefficient vector $q_3$ cannot be constructed from $Q$, i.e., $q_3$ is not a linear combination of $q_1$ and $q_2$; in fact, $q_3$ is orthogonal to both of them. Then, $q_3^T\tau$ need not inherit any of the weight assigned to $q_1^T\tau$ and $q_2^T\tau$; e.g., unlike $(\tau_3 - \tau_4)/\sqrt{2}$, which can clearly be constructed as a linear combination of $q_1^T\tau$ and $q_2^T\tau$. \hfill\textbardbl
\end{example}

For any system $Q^T\tau$, we propose to weight only the estimable functions that can be constructed from the functions of primary interest. Formally, one weights only $q^T\tau$, such that $q$ can be expressed as a linear combination of functions of primary interest, i.e., only if $q \in \C(Q)$. We note that the condition $q \in \C(Q)$ covers the original condition $q \in \E$, because $\C(Q) \subseteq \E$ since the original functions in $Q^T\tau$ need to be estimable too. In fact, the condition  $q \in \C(Q)$ reduces to $q \in \E$ when $r=\dim(\E)$, because $\C(Q)=\E$ then. That is a rather meaningful condition, as there is no point in assigning weights to non-estimable functions.

\subsection{Corresponding weight matrix}

The definition of the weight matrix presented in Section \ref{sExtendingWeights} allows us to provide the following weight matrix corresponding to any system of interest $Q^T\tau$:
For an arbitrary system of estimable functions $Q^T\tau$ (i.e., for any number of functions $s \geq 1$ of any achievable rank $r \leq \min\{s,\dim(\E)\}$), we define the unscaled weight matrix $W_Q= QQ^T$. If the particular functions of interest are assigned weights $b_1, \ldots, b_s$, then we simply consider the scaled system of functions of interest $\tilde{Q}^T\tau=B^{1/2}Q^T\tau$, resulting in $W_{\tilde{Q}}=QBQ^T$.

\begin{definition}
	Let $Q^T\tau$ be a system of estimable functions with weights $b_1, \ldots, b_s$. Then, the corresponding unscaled weight matrix is $W_Q=QQ^T$ and the scaled weight matrix is $W_{\tilde{Q}}=QBQ^T$, where $\tilde{Q}=QB^{1/2}$.
\end{definition}

When $r=\dim(\E)$, the term $I-P_\tau$ ensures the non-singularity of the weight matrix $I-P_\tau + QQ^T$. However, the relaxation of the non-singularity condition of the weight matrices in this paper eliminates the need for the 'regularizing term' $I-P_\tau$. Therefore, we obtain the simple form $W_Q = QQ^T$. In particular, when $\E = \R^v$ and $r=v$, then $P_\tau = I$ and thus the matrices $I-P_\tau + QQ^T$ and $QQ^T$ coincide. On the other hand, when $\E = 1_v^\perp$ and $r=v-1$, the matrices do not coincide, because the weight matrix of \cite{StallingsMorgan} is non-singular, whereas the weight matrix defined here is of rank $v-1$, as the set of estimable functions has only $v-1$ 'degrees of freedom'.

The following proposition shows that the weight matrix proposed here is equivalent to the weight matrix considered by \cite{StallingsMorgan}, with respect to the implied weights of the estimable functions, when the latter is defined (i.e., when $r=\dim(\E)$). Therefore, we obtain a weight matrix $W=QQ^T$, which is equivalent to the weight matrix $W=I-P_\tau + QQ^T$, but is of a simpler form.

\begin{proposition}
	Let $Q^T\tau$ be a system of functions of interest, such that $r=\dim(\E)$. Then, the weight matrices $W_1 = QQ^T$ and $W_2 = (I-P_\tau) + QQ^T$ are estimation equivalent. That is, $q^T W_1^- q = q^T W_2^{-1} q$ for any $q \in \E$. 
\end{proposition}

\begin{proof}
	Recall that $W_2^{-1} = (I-P_\tau) + (QQ^T)^+$.
	Since $r=\dim(\E)$, the matrix $P_\S$ in Proposition \ref{pEstEquiv} becomes $P_\tau$ and thus 
	$$P_\tau W_2^- P_\tau = P_\tau (QQ^T)^+ P_\tau =  P_\tau W_1^- P_\tau,$$
	where the first equality follows from $W_2^- =W_2^{-1}$; and the expression $P_\tau (QQ^T)^- P_\tau$ does not depend on the choice of $(QQ^T)^-$, yielding the last equality. Then, the estimation equivalence follows from Proposition \ref{pEstEquiv}.
\end{proof}

Similarly as in Theorem \ref{tEigEquiv}, the weighting given by $W_Q = QQ^T$ is equivalent to directly considering the information matrix $N_Q$, for any eigenvalue-based criterion.

\begin{theorem}\label{tEigEquivQQT}
	Let $Q^T\tau$ be a system of estimable functions and let $\xi$ be a feasible design for $Q^T\tau$. Then, the information matrix $N_Q(\xi)$ for $Q^T\tau$ and the weighted information matrix $C_{W_Q}(\xi) = (K_{W_Q}^T C^-(\xi) K_{W_Q})^{-1}$ have the same non-zero eigenvalues, including multiplicities.
\end{theorem}

\begin{proof}
	The non-zero eigenvalues of the information matrix $N_Q(\xi) = (Q^T C^-(\xi) Q)^+$ are inverse to the eigenvalues of the matrix $Q^T C^-(\xi) Q$. The matrix $Q^T C^-(\xi) Q$ can be expressed as $Q^T C^-(\xi) Q = (C^{+1/2} Q)^T  (C^{+1/2} Q)$, which shows that it has the same non-zero eigenvalues as the matrix 
	$(C^{+1/2} Q) (C^{+1/2} Q)^T = C^{+1/2} QQ^T C^{+1/2}.$
	Because $QQ^T = W_Q = K_{W_Q}K_{W_Q}^T$, we have $C^{+1/2} QQ^T C^{+1/2} = C^{+1/2} K_{W_Q}K_{W_Q}^T C^{+1/2}$; and $C^{+1/2} K_{W_Q}K_{W_Q}^T C^{+1/2}$ has the same non-zero eigenvalues as $K_{W_Q}^T C^+ K_{W_Q}$. Observing that the non-zero eigenvalues of the matrix $K_{W_Q}^T C^+ K_{W_Q}$ are inverse to the non-zero eigenvalues of $C_{W_Q}(\xi)$ completes the proof.
\end{proof}

Let $W$ be a weight matrix. Then, a system of estimable functions that corresponds to $W$ can be constructed employing the matrix $W^{1/2}$, similarly to Theorem \ref{tInverseProblem}.

\begin{theorem}\label{tInverseProblem2}
Let $W$ be a weight matrix. Then, the matrix $W$ corresponds to the system of estimable functions $W^{1/2} \tau$. Moreover, given a design $\xi$ that is feasible for $W^{1/2}\tau$, the weighted information matrix $C_W(\xi)$ has the same eigenvalues (including multiplicities) as the information matrix $N_{W^{1/2}}(\xi)$ for $W^{1/2}\tau$.
\end{theorem}

\begin{proof}
	The system $W^{1/2}\tau$ is estimable, because $\C(W^{1/2})=\C(W) \subseteq \E$.
	Because $W^{1/2}$ satisfies $W^{1/2}(W^{1/2})^T=W$, the weight matrix $W$ corresponds to the system $W^{1/2}\tau$. Then, Theorem \ref{tEigEquivQQT} yields the identity of the non-zero eigenvalues. Moreover, both $C_W(\xi)$ and $N_{W^{1/2}}(\xi)$ are $v \times v$ matrices, which yields the same multiplicity of the zero eigenvalue.
\end{proof}

A proper weight matrix $W_Q$ should place equal weight on each of the normalized functions $q_i^T\tau$. In the following proposition, we show that this is satisfied for the proposed $W_Q$ as long as the coefficient vectors $q_1, \ldots, q_s$ are linearly independent, i.e., as long as $\mathrm{rank}(Q)=s$.

\begin{proposition}\label{pProperWeights}
	Let $Q^T\tau$ be a full-rank system of normalized estimable functions. Then, $Q^TW_Q^-Q=I_s$. If $b_1, \ldots, b_s$ are the weights assigned to $q_1^T\tau, \ldots, q_s^T\tau$, respectively, then $Q^TW_{\tilde{Q}}^-Q=B^{-1}$, where $\tilde{Q}=QB^{1/2}$.
\end{proposition}

\begin{proof}
	We have $Q^TW_Q^-Q = Q^T (QQ^T)^- Q = I_s$, because $Q^T (QQ^T)^- Q$ is a symmetric idempotent matrix of full rank, i.e., $I_s$. The second statement follows from the analogous observation that $\tilde{Q}^T (\tilde{Q}\tilde{Q}^T)^- \tilde{Q} = I_s$, which yields $B^{1/2} Q^TW_{\tilde{Q}}^-Q B^{1/2} =I_s$.
\end{proof}

Clearly, when the functions in the system $Q^T\tau$ are not normalized, the weight of $q_i$ is $\lVert q_i \rVert^2$, because it can be expressed as $q_i = \lVert q_i \rVert \bar{q_i}$, where $\bar{q_i}$ is unscaled.

When the functions $q_1, \ldots, q_s$ are not linearly independent, the matrix $W_Q$ does not place equal weights on all of them (and neither does the matrix $I-P_\tau + QQ^T$, as noted in Section \ref{ssWeightedOpt}). We will elaborate on this in Section \ref{sPrimaryAndSecondary}.

The condition that a proper weight matrix must weight all the unscaled functions of interest equally (given $B=I_s$) does not uniquely define a weight matrix, as we have seen for $W_1 = QQ^T$ and $W_2 = I-P_\tau + QQ^T$. However, such matrices need not even be estimation equivalent. An extreme case is the matrix $W=I_v - J_v/v$, which satisfies $W^+=W$ and $q^T W^- q = 1$ for any $q \in \E$ as long as $\E= 1_v^\perp$. Clearly, $q^T(I_v - J_v/v)q = q^T q = 1,$ for any normalized $q$. Obviously, such a weight matrix does not provide any information about the considered set of estimable functions.

Even when we restrict ourselves to the non-singular matrices, as in \cite{StallingsMorgan}, the condition $q_i^T W^{-1} q_i = 1$ for all $i=1,\ldots, s$, does not uniquely specify a weight matrix, as is shown in the following example.

\begin{example}\label{exMultipleWMatrices}
	Consider the model \eqref{eModelTreat} with $v=3$ and suppose that the objective is to compare the test treatments $2$ and $3$ with the control (the first treatment). That is, $s=2$, $q_1=(-1,1,0)^T/\sqrt{2}$ and $q_2=(-1,0,1)^T/\sqrt{2}$. Let
	$$
	W_1=\begin{bmatrix}
	3/2 & -1/2 & -1/2 \\ -1/2 & 1/2 & 0 \\ -1/2 & 0 & 1/2
	\end{bmatrix}
	\quad \text{with} \quad
	W_1^{-1}=\begin{bmatrix}
	2 & 2 & 2 \\ 2 & 4 & 2 \\ 2 & 2 & 4
	\end{bmatrix}
	$$
	and
	$$
	W_2=\begin{bmatrix}
	5/2 & -1 & -1 \\ -1 & 2/3 & 1/3 \\ -1 & 1/3 & 2/3
	\end{bmatrix}
	\quad \text{with} \quad
	W_2^{-1}=\begin{bmatrix}
	2 & 2 & 2 \\ 2 & 4 & 1 \\ 2 & 1 & 4
	\end{bmatrix}.
	$$
	Then, $Q^T W_1^{-1} Q = Q^T W_2^{-1} Q = I_2$. Consider $q=(0,-1,1)^T/\sqrt{2}$. Then, $(q^T W_1^{-1} q)^{-1} =1/2$ and $(q^T W_2^{-1} q)^{-1} =1/3$, which shows that $W_1$ and $W_2$ are not estimation equivalent. We have found two positive definite weight matrices that assign the required weights to the given functions, yet they assign different weights to other estimable functions. \hfill\textbardbl
\end{example}

Therefore, it is crucial that the proposed weight matrix, $W_Q=QQ^T$, has a meaningful interpretation with respect to the system $Q^T\tau$, other than assigning equal weights to each of the functions in $Q^T\tau$. Theorems \ref{tEigEquiv} and \ref{tEigEquivQQT} state that the eigenvalue-based weighted optimality criteria with the weight matrices $I-P_\tau + QQ^T$ and $QQ^T$ are equivalent to the optimality with respect to the well-established information matrix for $Q^T\tau$. This shows that the weight matrix proposed by \cite{StallingsMorgan} as well as the weight matrix proposed here are indeed relevant when the objective is to estimate $Q^T\tau$. Lemma 1 by \cite{StallingsMorgan} and Proposition \ref{pEigenVar} here show the relevance of the weight matrices with respect to the weighted variances. Furthermore, in Section \ref{ssElements}, we provide a straightforward interpretation of the elements of $W_Q=QQ^T$, which indicates the relevance of the proposed matrix.

\subsection{Elements of the weight matrix}\label{ssElements}

Let us denote the rows of $Q$ by $q_{*1}, \ldots, q_{*v}$; note that $q_{*i}$ are row vectors. Then, the elements of $W_Q$ and $W_{\tilde{Q}}$ satisfy
$$ W_Q(i,j) = q_{*i} q_{*j}^T = \sum_{k=1}^s Q_{ik}Q_{jk}, \quad W_{\tilde{Q}}(i,j) = \tilde{q}_{*i} \tilde{q}_{*j}^T = \sum_{k=1}^s b_k Q_{ik} Q_{jk} \quad i,j=1,\ldots,v,$$
where $\tilde{q}_{*1}, \ldots, \tilde{q}_{*v}$ are the rows of $\tilde{Q}=QB^{1/2}$.
Such a form has a straightforward interpretation. The $i$-th diagonal element of $W_Q$ indeed represents the weight of the $i$-th parameter of interest: it is the sum of squares of all the coefficients for the $i$-th parameter across all functions of interest $q_1^T\tau, \ldots, q_s^T\tau$. The element on the position $(i,j)$, $i \neq j$, represents the amount of interest in the comparison of the $i$-th and $j$-th effects, as it is the sum of all products of the coefficients for the $i$-th and the $j$-th parameters across all functions of interest. High negative values of $W_Q(i,j)$ suggest a significant interest in the actual comparison, e.g.,  $\tau_i-\tau_j$; high positive values suggest a significant interest in their combined effect, e.g.,  $\tau_i+\tau_j$.

For the scaled version $W_{\tilde{Q}}$, the coefficients in each function $q_i^T\tau$ are multiplied by the square of the corresponding weight $b_i^{1/2}$ first, and only then the sums of squares or products are calculated - thus assigning the $i$-th function the relative weight $b_i$, $i=1,\ldots,s$.


In the following example, we demonstrate the previous observations.

\begin{example}\label{exWeightCoefficients}
	Consider model \eqref{eModelTreat} with the experimental objective as in Example \ref{exMultipleWMatrices}. Then, the matrix $W_Q$ is
	$$
	W_Q = QQ^T = \frac{1}{2}\begin{bmatrix}
	2 & -1 & -1 \\ -1 & 1 & 0 \\ -1 & 0 & 1
	\end{bmatrix}.
	$$
	Disregarding the common factor $1/2$, the weight of the first treatment is $2$, because it is present in $Q^T\tau$ two times, whereas the other two treatments have analogously weight 1. Since the second and the third treatment cannot be found in the same contrast, the weight assigned to their comparison is zero; similarly, the negative 'weight' $-1$ represents amount of interest in comparing the effects of the first and of the second (third) treatments.
	
	
	Now, suppose that the contrast $(\tau_2-\tau_1)/\sqrt{2}$ still has the weight $b_1 =1$, but the second treatment comparison  $(\tau_2-\tau_1)/\sqrt{2}$ is given an increased weight $b_2 = 2$. Then, each element of the second contrast is multiplied by $\sqrt{2}$, resulting in
	$$
	W_{\tilde{Q}} = \tilde{Q}\tilde{Q}^T = \frac{1}{2}\begin{bmatrix}
	3 & -1 & -2 \\ -1 & 1 & 0 \\ -2 & 0 & 2
	\end{bmatrix},
	$$
	which places more weight on the first and the second treatment, and on their comparison, as it should, because $b_2=2$ represents greater interest in comparing the first and the third treatment. \hfill\textbardbl
\end{example}

\section{$E$- and $A$-optimality}

Now that we have introduced the concept of the weight matrix corresponding to a system of estimable functions, we can provide the interpretation of the selected weighted optimality criteria for the general weight matrix, not necessarily implied by a system of estimable functions $Q^T\tau$. In particular, we provide interpretations of the weighted criteria of $E$-optimality and $A$-optimality in terms of the weighted variances of the estimable functions, similar (except Proposition \ref{pAoptWeighted}) to those by \cite{StallingsMorgan}.

The value of the weighted $E$-optimality criterion, or $E_W$-optimality criterion in short, is $\Phi_{EW}(\xi)=\lambda_{\min}(C_W(\xi))$.

\begin{proposition}\label{pEoptWeighted}
	Let $\xi \in \Xi$ and let $W$ be a weight matrix. Then, the inverse of the value of $E_W$-optimality is equal to the largest weighted variance over all $q^T\tau$, such that $q \in \C(W)$.
\end{proposition}

\begin{proof}
	We have
	$$
	\lambda_{\min}(C_W(\xi))=(\lambda_{\max}(K_W^T C^-(\xi) K_W))^{-1} = \Big(\max_{x \in \R^d} \frac{x^T K_W^T C^-(\xi) K_W x}{x^Tx}\Big)^{-1}.
	$$
	The largest weighted variance over all $q^T\tau$, such that $q \in \C(W)$, is
	$$
	\max_{q \in \C(W)} \frac{q^T C^-(\xi) q}{q^T W^- q} = \max_{z \in \R^d} \frac{z^T K_W^T C^-(\xi) K_W z}{z^T K_W^T W^- K_W z} = \max_{z \in \R^d} \frac{z^T K_W^T C^-(\xi) K_W z}{z^T z},
	$$
	where the first equality follows from the fact that $\C(W)=\C(K_W)$, and thus for any $q \in \C(W)$ there exists $z \in \R^d$, such that $q=K_Wz$. Since $W=K_WK_W^T$, we have $K_W^T (K_WK_W^T)^- K_W = I_d$, as in the proof of Proposition \ref{pEigenVar}, which yields the second equality.
\end{proof}

From Proposition \ref{pEoptWeighted} it follows that any $E_W$-optimal design minimizes the largest weighted variance over all $q^T\tau$, such that $q \in \C(W)$.
\bigskip

The value of the weighted $A$-optimality criterion, or $A_W$-optimality criterion in short, is $\Phi_{AW}(\xi)=d(\mathrm{tr} (C_W^{-1}(\xi)))^{-1} = d(\mathrm{tr} (K_W^T C^-(\xi)K_W))^{-1} $, see Chapter 6 by \cite{puk}.

\begin{proposition}\label{pAoptWeighted}
	Let $\xi \in \Xi$ and let $W$ be a weight matrix. Then, the inverse of the value of $A_W$-optimality is equal to the average variance of any system of $d$ estimable functions $Q^T\tau$ that satisfy $W_Q = W$.
\end{proposition}

\begin{proof}
	We have
	$(\Phi_{AW}(\xi))^{-1} = \mathrm{tr}(K_W^T C^-(\xi) K_W)/d.$
	Let $Q^T\tau$ be a system of $d$ estimable functions, such that $QQ^T=W$. Then, the average weighted variance for $q_1^T\tau, \ldots, q_d^T\tau$ is
	$$
	\frac{1}{d}\sum_{i=1}^d (q_i^T W^- q_i)^{-1} \Var(\widehat{q_i^T\tau}) = \mathrm{tr}(D Q^T C^-(\xi) Q)/d,
	$$
	where $D=\diag((q_1^T W^- q_1)^{-1}, \ldots, (q_d^T W^- q_d)^{-1})$. Moreover, $Q^T W^- Q = Q^T (QQ^T)^- Q = I_d$, as in the proof of Proposition \ref{pEigenVar}. Thus, $D=I_d$ and the average weighted variance is equal to $\mathrm{tr}(Q^T C^-(\xi) Q)/d$.
	
	From $QQ^T = K_WK_W^T$ follows that $\C(Q) = \C(QQ^T) = \C(K_W)$; therefore, there exists a $d \times d$ matrix $Z$, such that $Q=KZ$. Then, $K_WK_W^T = QQ^T = K_WZZ^T K_W^T $. By pre- and post-multiplying the equation $K_WK_W^T = K_WZZ^TK_W^T$ by $K_W^T$ and $K_W$, respectively, we obtain $K_W^T K_W K_W^T K_W = K_W^T K_W ZZ^T K_W^T K_W$. Because $K_W$ is of full column rank, the matrix $K_W^T K_W$ is non-singular, which yields $ZZ^T = I_d$. Therefore, the average weighted variance is equal to
	$$
	\mathrm{tr}(Z^T K_W^T C^-(\xi) K_W Z)/d = \mathrm{tr}( K_W^T C^-(\xi) K_W Z Z^T)/d = \mathrm{tr}( K_W^T C^-(\xi) K_W)/d = (\Phi_{AW}(\xi))^{-1}.
	$$
\end{proof}

In particular, from Proposition \ref{pAoptWeighted} it follows that the $A_W$-optimal design minimizes the average variance for $K_W^T\tau$.

In the following proposition, we also provide the interpretation of $A_W$-optimality analogous to that by \cite{StallingsMorgan}. We say that two estimable functions $q_1^T\tau$ and $q_2^T\tau$ are orthogonal with respect to the weight matrix $W$, or $W$-orthogonal in short, if $q_1^T W^- q_2 = 0$.

\begin{proposition}\label{pAoptWeightedOrt}
	Let $\xi \in \Xi$ and let $W$ be a weight matrix. Then, the inverse of the value of $A_W$-optimality is equal to the average variance for any system of $d$ mutually $W$-orthogonal estimable functions $Q^T\tau$, such that $\C(Q) \subseteq \C(W)$.
\end{proposition}

\begin{proof}
	As in Proposition \ref{pAoptWeighted}, the average weighted variance for $W$-orthogonal estimable functions $q_1^T\tau, \ldots, q_d^T\tau$ is $\mathrm{tr}(D Q^T C^-(\xi) Q)/d $, with $Q$ and $D$ defined as in the proof of Proposition \ref{pAoptWeighted}.
	
	The $W$-orthogonality of the considered functions yields that $Q^T W^- Q = D^{-1}$, which is equivalent to $D^{1/2} Q^T W^- Q D^{1/2} = I_d$.
	From $\C(QD^{1/2}) = \C(Q) \subseteq \C(W) = \C(K_W)$ follows that there exists a $d \times d$ matrix $Z$, such that $QD^{1/2}=KZ$. Therefore, $I_d = D^{1/2} Q^T W^- Q D^{1/2} = Z^T K_W^T W^- K_W Z = Z^T Z$, because $K_W^T W^- K_W = I_d$ as in the proof of Proposition \ref{pEigenVar}. That is, $Z$ is an orthogonal matrix. Then, the average variance can be expressed as
	$$\begin{aligned}
	\mathrm{tr}(D Q^T C^-(\xi) Q)/d 
	&= \mathrm{tr}(D^{1/2} Q^T C^-(\xi) Q D^{1/2})/d = \mathrm{tr}(Z^T K_W^T C^-(\xi) K_W Z )/d \\
	&= \mathrm{tr}(K_W^T C^-(\xi) K_W Z Z^T)/d = (\Phi_A(\xi))^{-1}.
	\end{aligned}$$
\end{proof}

Clearly, the maximum number of functions that can be mutually $W$-orthogonal is $d$, as $Q^T W^- Q = I_s$ cannot be achieved by a $v \times s$ matrix $Q$ with $s>d$. Moreover, the condition $\C(Q) \subseteq \C(W)$ ensures that the estimable functions can be weighted by $W$. Therefore, Proposition \ref{pAoptWeighted} shows that any $A_W$-optimal design minimizes the average variance for any system of $W$-orthogonal estimable functions of maximal size that can be weighted by $W$.

\section{Primary and secondary weights}\label{sPrimaryAndSecondary}

\subsection{Inconsistency}

There is a seeming inconsistency in the weighted optimality: seemingly consistent systems of weights produce inconsistent weight matrices. When we use a set of functions $Q^T\tau$ to construct the treatment weights, these weights imply a weight for any function $q^T\tau$ in $\C(Q)$, representing the amount of interest in this function. However, in general, we cannot consistently use any subset of functions in $\C(Q)$ with their implied weights, other than $Q^T\tau$, as an equivalent way of describing the experimental objective, because then the weight matrix would change, and the implied weights for the original functions $Q^T\tau$ would be different. The example below demonstrates this.

\begin{example}\label{exInconsistent}
	Consider the setting of comparing two test treatments with a control, as in Example \ref{exMultipleWMatrices}, with weights $b_1 = b_2 = 1$. Let $q_3^T\tau = (\tau_3-\tau_2)/\sqrt{2}$. Then, $q_3 = Q h$, where $h=(-1,1)^T$, and $q_3^T\tau$ has weight inverse to $h^T Q^T W_Q^- Q h = h^T h = 2 $. Thus, we have the amount of interest in $q_3^T\tau$ represented by the weight $b_3=1/2$.
	
	Now, consider $q_1^T\tau$ and $q_3^T\tau$, together denoted as $\bar{Q}^T\tau$, with the corresponding weights $b_1=1$ and $b_3 = 1/2$, and let us denote $W_2 = \bar{Q} B \bar{Q}^T$. Then, $q_2 = \bar{Q} h_2$, where $h_2 = (1,1)^T$, which yields the weight of $q_2$ inverse to $h^T Q^T W_2^- Q h = h^T B^{-1} h = 3$. Therefore, the weight implied for $q_2$ is $1/3$, which is inconsistent with the original weight $b_2=1$. It follows that the system of contrasts $q_1^T\tau$ and $q_2^T\tau$ with weights $b_1$ and $b_2$ implies a weight $b_3$ for $q_3^T\tau$; but that is not the same as building a weight matrix from $q_1^T\tau$ and $q_3^T\tau$ with the consistent weights $b_1$ and $b_3$.
	
	In other words, although the original system of contrasts implies a weight $1/2$ for the contrast $q_3^T\tau$, we cannot actually use this weight to equivalently define the same system of weights.  \hfill\textbardbl
\end{example}

Therefore, given a system of estimable functions $Q^T\tau$, we propose to distinguish between the \emph{primary weights} $b_1, \ldots, b_s$ of the original functions of interest, and the \emph{secondary} (or \emph{implied}) \emph{weights} $w(q) = (q^T W_{\tilde{Q}}^- q)^{-1}$  given by the weight matrix $W_{\tilde{Q}}$ for any function $q^T\tau$ that satisfies $q \in \C(Q)$. These two systems of weights are not equivalent: the primary weights are pre-specified for a finite number of functions of interest, and they represent the primary interest assigned to these functions; the secondary weights represent the amount of interest in any $q \in \C(Q)$ implied by the functions of interest and their primary weights.

In Example \ref{exInconsistent}, the original weights $b_1=b_2=1$ are the primary weights for $q_1^T\tau$ and $q_2^T\tau$, respectively. By contrast, the weight $w(q_3)=1/3$ of $q_3^T\tau$ is not a primary weight, but one that is implied by the two functions of interest.

\subsection{Rank-deficient systems}

For a rank-deficient system of functions of interest, we may wish to assign the same weight to each of the functions of interest, similarly as in Proposition \ref{pProperWeights} for the full rank systems. However, in general, this does not happen with the proposed weight matrix $QQ^T$, nor with the matrix $I-P_\tau + QQ^T$ proposed by \cite{StallingsMorgan} (as noted in Section \ref{ssWeightedOpt}). Consider the extreme case, where $Q=(q_1,q_2)$, where $q_1=q_2$, with $b_1=b_2=1$, i.e., we have two functions of interest, which are identical. Then, it is easy to see that $QQ^T = 2q_1q_1^T$ and $q_i^T (QQ^T)^- q_i = 1/2$ for $i=1,2$, which yields the \emph{secondary} weights of $q_1^T\tau$ and $q_2^T\tau$ to be $w(q_1)=w(q_2)=2$.

Without differentiating between the primary and the implied weights, such a result might seem nonsensical - because the function $q_1^T\tau$ is assigned the weight $b_1=1$, and yet the corresponding weighting by $W_Q$ would at the same time provide double the weight for $q_1^T\tau$. However, since the function $q_1^T\tau$ is present twice in the system of functions of interest, it is natural that its \emph{implied} weight $w(q_2)$ is 2, unlike its primary weight $b_1=1$.

Let us return to Example \ref{exInconsistent}.

\begin{example}
	Consider the setting as in Example \ref{exInconsistent}, and let $Q=(q_1, q_2, q_3)$ with $b_1=b_2=1$ and $b_3=1/2$. That is, we have added the contrast $q_3^T\tau$ to the original system of two contrasts, with its weight given by the first two contrasts, see Example \ref{exInconsistent}. Then,
	$$
	\tilde{Q}=\frac{1}{\sqrt{2}}\begin{bmatrix}
	-1 & -1 & 0 \\ 1 & 0 & -1/\sqrt{2} \\ 0 & 1 & 1/\sqrt{2}
	\end{bmatrix}, \text{ and }
	G=\begin{bmatrix}
	3/2 & 1 & 0 \\ 1 & 2 & 0 \\ 0 & 0 & 0
	\end{bmatrix}
	$$
	is a generalized inverse of $\tilde{Q}\tilde{Q}^T$. Then, it is easy to calculate that $w(q_3) = (q_3^T G q_3)^{-1} = 1$, and $w(q_1) = w(q_2) = 4/3$. Since we have added a third contrast, which is confounded with the first two contrasts, the third contrast provides some secondary weight for the other two, and vice versa. Then, it is to be expected that the implied weight of each of $q_1^T\tau, q_2^T\tau, q_3^T\tau$ is greater than its primary weight, not even preserving the relative weights. \hfill\textbardbl
\end{example}

We have demonstrated that in the rank-deficient case a statement analogous to Proposition \ref{pProperWeights} does not hold. In fact, employing the notion of the secondary weights, in general such a statement should not hold: i.e., given $B=I_s$, all of the secondary weights implied by $W_Q$ should not always be equal to 1. Because of the confounding in the rank-deficient systems, the functions of interest provide each other some of their weight, resulting in secondary weights different from the primary ones.

Intuitively, the secondary weights should be at least equal to the primary ones, because of the 'sharing' of the weights. This is demonstrated in the following proposition.

\begin{proposition}
	Let $Q^T\tau$ be a system of estimable functions with weights $b_1 ,\ldots, b_s$, and let  $W_{\tilde{Q}}$ be the corresponding weight matrix. Then, the secondary weight of any of the functions $q_i^T\tau$ with respect to $W_{\tilde{Q}}$ is at least $b_i$, i.e., $w(q_i) \geq b_i$, $i=1,\ldots, s$. Moreover, if $q_i$ can be expressed as a linear combination of the other columns of $Q$, then $w(q_i) > b_i$.
\end{proposition}

\begin{proof}
	First, consider the unscaled version $W_Q = QQ^T$. The weight of $q_i^T\tau$ is inverse to $q_i^T (QQ^T)^- q_i = e_i^T Q^T (QQ^T)^- Q e_i$, where $e_i$ is the $i$-th column of $I_s$. The matrix $P_Q:= Q^T (QQ^T)^- Q$ is idempotent, therefore $\rho(P_Q)=1$, where $\rho$ is the spectral norm of a matrix. Then,
	$$e_i^T P_Q e_i = e_i^T P_Q P_Q e_i = \lVert P_Q e_i \rVert_2^2 \leq \rho^2(P_Q) \lVert e_i \rVert_2^2 = 1.$$
	Therefore, $w(q_i) \geq 1$.
	Now, consider $W_{\tilde{Q}} = \tilde{Q}\tilde{Q}^T$. Then, $\tilde{q}_i^T (\tilde{Q}\tilde{Q}^T)^- \tilde{q}_i \leq 1$, the same as for $W_Q$. From $\tilde{q}_i = b_i^{1/2} q_i$ follows that $b_i w^{-1}(q_i) \leq 1$, which implies $w(q_i) \geq b_i$.
	
	Note that equality is attained in $\lVert P_Q e_i \rVert_2^2 \leq \rho^2(P_Q) \lVert e_i \rVert_2^2$ if and only if $e_i \in \arg\max_x \lVert P_Q x \rVert_2/\lVert x \rVert_2$, i.e., if and only if $e_i$ is an eigenvector corresponding to $\lambda_1(P_Q) = 1$. The last condition is equivalent to $e_i \in \C(Q^T)$, because $P_Q$ is the projector on $\C(Q^T)$. 
	
	Suppose that $q$ can be expressed as a linear combination of the other columns of $Q$, $q_i = \sum_{j \neq i} \alpha_j q_j$. Moreover, suppose that $w(q_i) =b_i$, i.e., $\lVert P_Q e_i \rVert_2^2 = \rho^2(P_Q) \lVert e_i \rVert_2^2$, which is equivalent to $e_i \in \C(Q^T)$. Then there exists a $z \in \R^v$, such that $Q^T z = e_i$, i.e., $q_i^T z = 1$ and $q_j^T z = 0$ for $j \neq i$. But then $q_i^T z = \sum_{j \neq i} \alpha_j q_j^T z = 0$, which contradicts $q_i^T z = 1$.
\end{proof}

\section{Discussion}

If the objective of the experiment is to estimate a system of estimable functions $Q^T\tau$, possibly with weights $b_1, \ldots, b_s$, and an eigenvalue-based optimality criterion is considered, then the theory of the weighted optimality criteria is equivalent to the theory of optimality for $Q^T\tau$ (using $N_Q(\xi)$), as shown in Theorems \ref{tEigEquiv} and \ref{tEigEquivQQT}. However, the latter approach may be more appropriate, as it does not require the construction of the weight matrix and it is a well-developed and thoroughly justified theory tailored for estimating a system of estimable functions. When some non-uniform weights $b_i$ are considered, the weighting can be simply done by rescaling $Q$ to $QB^{1/2}$, i.e., by multiplying each coefficient vector $q_i$ by $b_i^{1/2}$.

When the weighted optimality is chosen, we have extended its scope, so that it includes the systems of estimable functions with rank less than $\dim(\E)$, as well as singular weight matrices. Consequently, the weight matrix $W_Q=QQ^T$ can be used instead of that by \cite{StallingsMorgan}. Such a weight matrix is of a simpler form, and has a straightforward interpretation of its elements (see Section \ref{ssElements}).

We noted that it is necessary to differentiate between the primary and the secondary (implied) weights, as they are not equivalent. In fact, given a set of functions of interest $Q^T\tau$, the calculation of the implied weights for other estimable functions allows for an alternative use of the weight matrices. Instead of being an alternative to the 'standard' optimality for $Q^T\tau$, the corresponding weight matrix $W_Q$ can be viewed as an addition to the optimality theory for $Q^T\tau$. It provides a tool for analysis of the functions of interest and other estimable functions: the corresponding weight matrix $W_Q$ allows one to calculate the secondary weights, which measure the implied amount of interest in any estimable function $q^T\tau$ satisfying $q \in \C(Q)$, given the original system of (weighted) estimable functions $Q^T\tau$. Such an analysis was performed in Section 4.2 by \cite{StallingsMorgan}.

\bibliographystyle{plainnat}
\bibliography{rosa.bib}

\end{document}